\documentclass{amsart}
\usepackage{amsthm}
\usepackage{amssymb}
\usepackage{mathtools}
\usepackage{graphicx}
\usepackage{enumitem}
\usepackage[all]{xy}
\usepackage{pgfplots}
\usetikzlibrary{arrows}
\usepackage{array}
\usepackage{caption}
\usepackage[normalem]{ulem}

\definecolor{lightgrey}{rgb}{.804,.804,.756}
\usepackage[pdfauthor={V. Lebed},
pdftitle={The word problem for Hecke--Kiselman monoids}, 
colorlinks=false,linkbordercolor=lightgrey,citebordercolor=lightgrey,urlbordercolor=lightgrey]{hyperref}

\definecolor{myred}{rgb}{.545,0,0}
\definecolor{myblue}{rgb}{.024,.15,.645}    
\definecolor{mygreen}{rgb}{0,.455,0} 
\definecolor{mygrey}{rgb}{.8,.8,.8}
\definecolor{uuuuuu}{rgb}{0,0,0}
\definecolor{ccqqqq}{rgb}{0.8,0.,0.}
\definecolor{qqqqcc}{rgb}{0.,0.,0.8}    
\definecolor{dcrutc}{rgb}{0.86,0.09,0.24}
\definecolor{ffttcc}{rgb}{1.,0.2,0.8}
\definecolor{wwzzff}{rgb}{0.4,0.6,1.}
\definecolor{qqttzz}{rgb}{0.,0.2,0.6}   

\makeatletter
	\theoremstyle{plain}
\newtheorem{thm}{Theorem}[section]
\newtheorem{lem}[thm]{Lemma}

\newtheorem{pro}[thm]{Proposition}
\newtheorem*{conjecture*}{Conjecture}

	\theoremstyle{definition}
\newtheorem{defn}[thm]{Definition}

\newtheorem{exa}[thm]{Example}
	\theoremstyle{remark}
\newtheorem{rem}[thm]{Remark}

\setlist{nolistsep}
\setenumerate{topsep=0pt}
\makeatother

\newcommand{\NN}{\mathbb{N}}
\newcommand{\ZZ}{\mathbb{Z}}
\newcommand{\RR}{\mathbb{R}}
\newcommand{\SSS}{\mathbb{S}}
\newcommand{\wA}{\widetilde{A}}
\newcommand{\HK}{\operatorname{HK}}
\newcommand{\HKC}{\mathsf{C}_n}
\newcommand{\HKL}{\mathsf{L}_n}
\newcommand{\DHKC}{\mathsf{D}\mathsf{C}_n}
\newcommand{\DHKL}{\mathsf{D}\mathsf{L}_n}
\newcommand{\WHKC}{\mathsf{W}\mathsf{C}_n}
\newcommand{\WHKL}{\mathsf{W}\mathsf{L}_n}
\newcommand{\Id}{\operatorname{Id}}
\newcommand{\II}{\operatorname{II}_n}
\newcommand{\IIC}{\operatorname{IIC}_n}
\newcommand{\stkout}[1]{\ifmmode\text{\sout{\ensuremath{#1}}}\else\sout{#1}\fi}
\newcommand{\Perm}{S_{n+1}^{\stkout{321}}}
\newcommand{\geqrot}{\mathbin{\rotatebox[origin=c]{-90}{$\geqslant$}}}

\begin{document}

\title[]{The word problem for Hecke--Kiselman monoids \\ of type $A_n$ and $\widetilde{A}_n$}

\begin{abstract}
We exhibit explicit and easily realisable bijections between Hecke--Kiselman monoids of type $A_n$/$\widetilde{A}_n$; certain braid diagrams on the plane/cylinder; and couples of integer sequences of particular types. This yields a fast solution of the word problem and an efficient normal form for these HK monoids. Yang--Baxter type actions play an important role in our constructions.
\end{abstract}



\author{Victoria Lebed}
\address{LMNO, Universit\'e de Caen--Normandie, BP 5186, 14032 Caen Cedex, France}
\email{lebed@unicaen.fr}

\maketitle

\section{Introduction}

In this paper we study the word problem for two closely related monoids $\HKL$ and $\HKC$ (with $n \in \NN$), defined by generators and relations as follows. The generators are $x_i, \ 1 \leqslant i \leqslant n$, and the relations are
\begin{align}
&x_i^2=x_i, &&1 \leqslant i \leqslant n, \label{E:idempot}\\
&x_ix_j=x_jx_i, && 1< i-j <n(-1),\label{E:FarComm}\\
&x_ix_{i+1}x_i = x_{i+1}x_ix_{i+1} = x_ix_{i+1}, &&1 \leqslant i < n (+1).\label{E:braid}
\end{align}
Both $(-1)$ and $(+1)$ are omitted for $\HKL$, and preserved for $\HKC$. Thus to get $\HKC$ from $\HKL$ one replaces one far-commutativity relation \eqref{E:FarComm} with one braid-like relation \eqref{E:braid}. For the generators of $\HKC$, the subscripts are taken modulo $n$; thus, $x_{n+1}$ means  $x_1$. 

These monoids are particular cases of Hecke--Kiselman monoids $\HK_{\Theta}$ \cite{GaMa}, defined for any partially oriented graph $\Theta$. The monoid $\HKL$ corresponds to the linearly oriented chain $A_n$ (hence the L in the name we chose), and $\HKC$ to the linearly oriented cycle $\wA_n$ (hence the C in the name, which goes back at least to \cite{MeOk}). HK monoids are quotients of $0$-Hecke monoids (to get those, remove the $= x_ix_{i+1}$ part from the relations), which are themselves quotients of Artin--Tits monoids, useful in the study of the representation theory thereof. The $= x_ix_{i+1}$ part of \eqref{E:braid} comes from Kiselman monoids from convexity theory \cite{Ki,KuMa}. HK monoids have applications to computer simulations via discrete sequential dynamical systems \cite{CoDA}; and to the representation theory of the path algebra of the quiver $\Theta$ via projection functors \cite{Gre,GreMa,GreMa2}.

As shown in \cite{So}, and later more explicitly and with a simpler proof in \cite{GaMa}, $\HKL$ is isomorphic to the Catalan monoid $CM_{n+1}$. This is the monoid of all order-preserving order-decreasing transformations of the set $\{1, 2, \ldots, n + 1\}$, appearing in diverse combinatorial and representation-theoretic contexts. This solves the word problem for $\HKL$, and identifies its size as the Catalan number $C_{n+1}=\frac{1}{n+2}\  {{2n+2}\choose{n+1}}$. 

In this paper we establish bijections between:

\newpage
\begin{enumerate}[label=(\roman*)]
\item the elements of $\HKL$;
\item\label{I:i1} positive braids on $n+1$ strands without bigons nor triangles (that is, any two strands intersect at most once, and no triples of strands intersect pairwise);
\item\label{I:ii1} increasing couples of increasing integer sequences bounded by $1$ and $n+1$, that is, $2k$ integers, for any $0 \leqslant k \leqslant n$, satisfying the inequalities
\begin{center}
\begin{tabular}{*{11}{>$c<$}}
 &  & b_1 & < & b_2 & < & \ldots & < & b_k & \leqslant & n+1 \\
 &  & \vee &  & \vee &  & \ldots &  & \vee & & \\
 1 & \leqslant & a_1 & < & a_2 & < & \ldots & < & a_k & & 
\end{tabular}
\end{center}
\item the set of $321$-avoiding permutations from $S_{n+1}$.
\end{enumerate}
Our bijections are explicit, and proofs are self-contained. We also recall a bijection between the sequences from \ref{I:ii1} and the Catalan monoid $CM_{n+1}$, which reflects a deeper connection between the two \cite{MarStein}. The cardinality of $CM_{n+1}$ being $C_{n+1}$, we recover not-so-widely-known interpretations of Catalan numbers, given by counting the braids from \ref{I:i1}, the sequences from \ref{I:ii1}, or $321$-avoiding permutations. 

The monoids $\HKL$ serve us as a toy model for studying the $\HKC$, crucial for understanding $\HK_{\Theta}$ for a general graph $\Theta$. We thus start with the simpler case $\HKL$, which makes often technical constructions for $\HKC$ more intuitive. 

The word problem for $\HKC$ was solved in \cite{MeOk} by exhibiting a finite Gr\"{o}bner basis. That solution was reformulated in terms of confluent reductions in \cite{ArDA19}. Further, the reduced form (with respect to the Gr\"{o}bner basis from \cite{MeOk}) of almost all the elements of $\HKC$ was given in \cite{OkWie}. This was used to show that the algebra $K\left[\HKC\right]$, where $K$ is a field, is Noetherian, and then to classify all graphs $\Theta$ for which $K\left[\HK_{\Theta}\right]$ is Noetherian.

Our main results are the following bijections, inspired by the case of $\HKL$:
\begin{enumerate}[label=(\roman*)]
\item\label{I:i} the elements of $\HKC$;
\item\label{I:ii} positive braids on $n$ strands on a cylinder generated by the elementary crossings, without contractible bigons nor contractible triangles (cf. Fig.~\ref{P:Contractible});
\item\label{I:iii} $n$-close increasing couples of increasing integer sequences, that is, $2k$ integers, for any $0 \leqslant k < n$, satisfying the inequalities
\begin{center}
\begin{tabular}{*{11}{>$c<$}}
 &  & b_1 & < & b_2 & < & \ldots & < & b_k & < & b_1+n \\
 &  & \vee &  & \vee &  & \ldots &  & \vee & & \\
 1 & \leqslant & a_1 & < & a_2 & < & \ldots & < & a_k & \leqslant & n
\end{tabular}
\end{center}
\end{enumerate}
To construct such a sequence couple for a word $x$ in the generators $x_i$, we develop an algorithm linear in the number of letters in $x$. This efficiently solves the word problem for $\HKC$. In the opposite direction, from such a sequence couple we deduce a word in the $x_i$, yielding preferred representatives of all the elements of $\HKC$, different from the reduced forms from \cite{OkWie}. The diagrammatic interpretation \ref{I:ii} is less useful in practice, but crucial in our proof of the bijection \ref{I:i} $\leftrightarrow$ \ref{I:iii}.

To compare elements in $\HKL$ or $\HKC$, we make these monoids act on the powers of certain sets, generalising the action from \cite{ArDA13} crucial in \cite{OkWie}. Multiple examples of such actions are given.

\section{Yang--Baxter like actions}

In this section we make the monoid $\HKL$ or $\HKC$ act on the powers $A^{n+1}$ or $A^n$ of a set $A$ locally, that is, the generator $x_i$ affects only the components $i$ and $i+1$ of $A^{\bullet}$. This generalisation of the actions by Yang--Baxter operators was considered, in the case of braid groups, in \cite{Ito}. Diverse examples are given. One of them will further be shown to be faithful, and play a key role in our arguments.

\begin{defn}
An \emph{$\HKL$-chain} (resp., \emph{$\HKC$-chain}) on a set $A$ is a collection of idempotent maps $\sigma_i \colon A \times A \to A \times A$, $i=1,\ldots,n$, satisfying the relation
\begin{align}\label{E:Chain}
(\sigma_i \times \Id_A)(\Id_A \times \sigma_{i+1})(\sigma_i \times \Id_A)
&=(\Id_A \times \sigma_{i+1})(\sigma_i \times \Id_A)(\Id_A \times \sigma_{i+1})\\
&=(\sigma_i \times \Id_A)(\Id_A \times \sigma_{i+1})\notag
\end{align}
for all $1 \leqslant i < n$ (resp., for all $1 \leqslant i \leqslant n$). As usual, we set $\sigma_{n+1}=\sigma_1$.
\end{defn}

When all the $\sigma_i$s coincide ($=\sigma$), one recovers the notion of an idempotent Yang--Baxter operator satisfying the Kiselman property $(\Id_A \times \sigma)(\sigma \times \Id_A)(\Id_A \times \sigma)=(\sigma \times \Id_A)(\Id_A \times \sigma)$.

\begin{pro}
Let $\sigma_1, \ldots, \sigma_n$ be an $\HKL$-chain (resp., $\HKC$-chain) on $A$. Then an action of the monoid $\HKL$ (resp., $\HKC$) on $A^{n+1}$ (resp., $A^n$) can be defined as follows: 
\begin{align*}
x_i &\mapsto \Id_A^{i-1} \times \sigma_i \times \Id_A^{n-i} \qquad \text{for all }i
\end{align*}
in the case $\HKL$, and 
\begin{align}
x_i &\mapsto \Id_A^{i-1} \times \sigma_i \times \Id_A^{n-i-1} \qquad \text{for all }i<n,\label{E:usual}\\
x_n &\mapsto \theta^{-1}(\sigma_n \times \Id_A^{n-2})\theta\label{E:last}
\end{align}
in the case $\HKC$. Here the map $\theta$ permutes the components of $A^n$ by moving the last component to the beginning.
\end{pro}

Despite appearances, the assignment \eqref{E:last} is of the same nature as \eqref{E:usual}: arranging the components of $A^n$ on the circle rather than a line, it says that $\sigma_n$ should be applied to components $n$ and $1$ of $A^n$, which now become neighbours.

\begin{proof}
Relations \eqref{E:idempot}--\eqref{E:braid} hold true by an easy inspection.
\end{proof}

\begin{exa}
The identities $\sigma_i(a,b)=(a,b)$ yield an $\HKL$- or an $\HKC$-chain on any set.
\end{exa}

\begin{exa}
Generalising the preceding example, one can put $\sigma_i(a,b)=(a,p_i(b))$, where $p_i\colon A \to A$ are any projectors: $p_i^2=p_i$.
\end{exa}

\begin{exa}\label{EX:NiceAction}
Consider a set $A$ and $n$ maps $f_i \colon A \to A$. Then the maps $\sigma_i(a,b)=(a,f_i(a))$ form an $\HKL$-chain and a $\HKC$-chain. Indeed, the idempotence is obvious, and the three parts of \eqref{E:Chain} applied to a triple $(a,b,c)$ all yield $(a, f_i(a), f_{i+1}f_i(a))$. Taking $A=\ZZ$, $f_i=\Id_{\ZZ}$ for $i<n$ and $f_n(a)=a+1$, we get a mirror version of the $\HKC$-actions from \cite[Proof of Lemma 2.6]{ArDA13}. Taking $A=\ZZ$, $f_i=\Id_{\ZZ}$ for $i \leqslant n$, and considering the action on the element $(1,2,\ldots,n,n+1) \in A^{n+1}$, one recovers the isomorphism between $\HKL$ and the Catalan monoid from \cite{GaMa}.
\end{exa}

\begin{exa}
Consider a monoid $A$ and $n$ monoid homomorphisms $f_i \colon A \to A$. Then the maps $\sigma_i(a,b)=(1,f_i(a)b)$ form an $\HKL$-chain and a $\HKC$-chain. Indeed, the idempotence follows from $f_i(1)=1$, and the three parts of \eqref{E:Chain} applied to a triple $(a,b,c)$ all yield $(1, 1, f_{i+1}f_i(a)f_{i+1}(b)c)$. When all the $f_i$ are equal, one recovers the idempotent Yang--Baxter operators from \cite{StVo}, which generalise the $f_i=\Id_A$ case from \cite{Leb}.
\end{exa}

\begin{exa}
Let the set $A$ be endowed with an associative operation $*$ satisfying the absorption property $a*(a*b)=a*b$. The set $\ZZ$ with the operation $\min$ or $\max$ is an elementary example. Then the maps $\sigma_i(a,b)=(a,a*b)$ form an $\HKL$-chain and a $\HKC$-chain. Indeed, the idempotence follows from the absorption property, and the three parts of \eqref{E:Chain} applied to a triple $(a,b,c)$ all yield $(a, a*b, a*b*c)$.
\end{exa}

\section{From HK monoids to diagrams}

This section presents a diagrammatic version of the monoids $\HKL$ and $\HKC$, inspired by the classical braid interpretation of braid monoids. As in all subsequent sections, we start with the more intuitive $\HKL$ case, and then adapt our arguments to the $\HKC$.

By an \emph{$n$-diagram} we mean a continuous map $d \colon \bigsqcup_{i=1}^{n+1} I_i \to \RR \times I$ sending $n+1$ disjoint copies of the unit intervals $I=[0,1]$ to the unit strip, in such a way that:
\begin{enumerate}[label=(\Alph*)]
\item\label{it:first} the vertical projection sends each \emph{strand} $d(I_i)$ bijectively onto $I$;
\item $d$ sends the endpoints $0,1 \in I_i$ to $(k,0)$ and $(l,1)$ respectively, for some $k,l \in \{1,2,\ldots,n+1\}$;
\item\label{it:last} $d$ is injective except for a finite number of interior double points.
\end{enumerate}
See Fig.~\ref{P:DiagsEx} for examples. These $n$-diagrams are considered up to \emph{$n$-diagram isotopy}, with the usual notion of isotopy for topological objects. They can be thought of as braid or, alternatively, $4$-valent graph diagrams.

\begin{figure}[h]
\centering
\begin{tikzpicture}[x=15,y=15]
\draw [line width=1pt,color=ccqqqq] (0.,0.)-- (2.,4.);
\draw [line width=1pt,color=qqqqcc] (1.,0.)-- (0.,4.);
\draw [line width=1pt,color=ccqqqq] (2.,0.)-- (3.,4.);
\draw [line width=1pt,color=qqqqcc] (3.,0.)-- (1.,4.);
\draw [line width=1pt] (4.,0.)-- (4.,4.);
\begin{scriptsize}
\draw [fill=ccqqqq] (0.,0.) circle (1.0pt);
\draw[color=ccqqqq] (0,-0.31) node {$1$};
\draw [fill=qqqqcc] (1.,0.) circle (1.0pt);
\draw[color=qqqqcc] (1,-0.31) node {$2$};
\draw [fill=ccqqqq] (2.,0.) circle (1.0pt);
\draw[color=ccqqqq] (2,-0.31) node {$3$};
\draw [fill=qqqqcc] (3.,0.) circle (1.0pt);
\draw[color=qqqqcc] (3,-0.31) node {$4$};
\draw [fill=uuuuuu] (4.,0.) circle (1.0pt);
\draw[color=uuuuuu] (4,-0.31) node {$5$};
\draw [fill=qqqqcc] (0.,4.) circle (1.0pt);
\draw [fill=qqqqcc] (1.,4.) circle (1.0pt);
\draw [fill=ccqqqq] (2.,4.) circle (1.0pt);
\draw [fill=ccqqqq] (3.,4.) circle (1.0pt);
\draw [fill=uuuuuu] (4.,4.) circle (1.0pt);
\end{scriptsize}
\end{tikzpicture} \qquad\qquad\qquad
\begin{tikzpicture}[x=15,y=15]
\draw [line width=1.pt] (3.,0.)-- (4.,4.);
\draw [line width=1.pt] (4.,0.)-- (1.,4.);
\draw [line width=1.pt] (1.,0.)-- (3.,4.);
\draw [shift={(-1.7084782608695648,2.)},line width=1.pt]  plot[domain=-0.863846371292726:0.8638463712927259,variable=\t]({1.*2.6303798143735615*cos(\t r)+0.*2.6303798143735615*sin(\t r)},{0.*2.6303798143735615*cos(\t r)+1.*2.6303798143735615*sin(\t r)});
\draw [shift={(2.4712658227848103,2.)},line width=1.pt]  plot[domain=1.8022079906921982:4.4809773164873885,variable=\t]({1.*2.054772852586155*cos(\t r)+0.*2.054772852586155*sin(\t r)},{0.*2.054772852586155*cos(\t r)+1.*2.054772852586155*sin(\t r)});
\begin{scriptsize}
\draw [fill=uuuuuu] (0.,0.) circle (1.0pt);
\draw [fill=uuuuuu] (1.,0.) circle (1.0pt);
\draw [fill=uuuuuu] (2.,0.) circle (1.0pt);
\draw [fill=uuuuuu] (3.,0.) circle (1.0pt);
\draw [fill=uuuuuu] (4.,0.) circle (1.0pt);
\draw [fill=uuuuuu] (0.,4.) circle (1.0pt);
\draw [fill=uuuuuu] (1.,4.) circle (1.0pt);
\draw [fill=uuuuuu] (2.,4.) circle (1.0pt);
\draw [fill=uuuuuu] (3.,4.) circle (1.0pt);
\draw [fill=uuuuuu] (4.,4.) circle (1.0pt);
\draw[color=uuuuuu] (0,-0.31) node {$1$};
\draw[color=uuuuuu] (1,-0.31) node {$2$};
\draw[color=uuuuuu] (2,-0.31) node {$3$};
\draw[color=uuuuuu] (3,-0.31) node {$4$};
\draw[color=uuuuuu] (4,-0.31) node {$5$};
\end{scriptsize}
\end{tikzpicture} \qquad\qquad\qquad
\begin{tikzpicture}[x=15,y=15]
\begin{scriptsize}
\draw [line width=1.pt] (4.,0.)-- (1.,4.);
\draw [line width=1.pt] (1.,0.)-- (4.,4.);
\draw [line width=1.pt] (3.,0.)-- (0.,4.);
\draw [shift={(-1.535,3.2675)},line width=1.pt]  plot[domain=-1.131617191598453:0.20432197359684062,variable=\t]({1.*3.6100943547226017*cos(\t r)+0.*3.6100943547226017*sin(\t r)},{0.*3.6100943547226017*cos(\t r)+1.*3.6100943547226017*sin(\t r)});
\draw [shift={(3.33008,1.79248)},line width=1.pt]  plot[domain=1.7192219478353934:4.074006033090464,variable=\t]({1.*2.2320612350023015*cos(\t r)+0.*2.2320612350023015*sin(\t r)},{0.*2.2320612350023015*cos(\t r)+1.*2.2320612350023015*sin(\t r)});
\
\draw [fill=uuuuuu] (0.,0.) circle (1.0pt);
\draw [fill=uuuuuu] (1.,0.) circle (1.0pt);
\draw [fill=uuuuuu] (2.,0.) circle (1.0pt);
\draw [fill=uuuuuu] (3.,0.) circle (1.0pt);
\draw [fill=uuuuuu] (4.,0.) circle (1.0pt);
\draw [fill=uuuuuu] (0.,4.) circle (1.0pt);
\draw [fill=uuuuuu] (1.,4.) circle (1.0pt);
\draw [fill=uuuuuu] (2.,4.) circle (1.0pt);
\draw [fill=uuuuuu] (3.,4.) circle (1.0pt);
\draw [fill=uuuuuu] (4.,4.) circle (1.0pt);
\draw[color=uuuuuu] (0,-0.31) node {$1$};
\draw[color=uuuuuu] (1,-0.31) node {$2$};
\draw[color=uuuuuu] (2,-0.31) node {$3$};
\draw[color=uuuuuu] (3,-0.31) node {$4$};
\draw[color=uuuuuu] (4,-0.31) node {$5$};
\end{scriptsize}
\end{tikzpicture}
\caption{Examples of $4$-diagrams}\label{P:DiagsEx}
\end{figure}
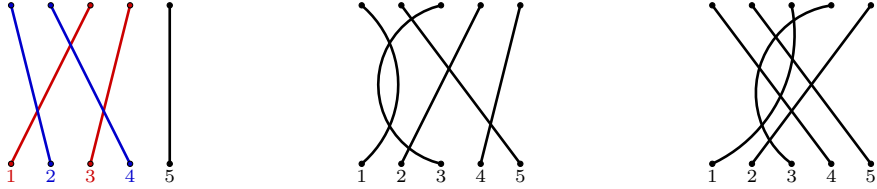

We also need \emph{$2$-} and \emph{$3$-moves} for $n$-diagrams, depicted on Fig.~\ref{P:Moves}. These moves are local: they involve only a small region of the diagram, outside of which the diagram remains unchanged. These moves are reminiscent of the Reidemeister moves $2$ and $3$, and are motivated by the Hecke--Kiselman relations.

\begin{figure}[h]
\centering
\begin{tikzpicture}[x=15,y=15]
\fill[line width=.8pt,dash pattern=on 5pt off 5pt,fill=white,draw=black] (-0.5,2.) -- (1.5,2.) -- (1.5,0.) -- (-0.5,0.) -- cycle;
\fill[line width=.8pt,dash pattern=on 5pt off 5pt,fill=white,draw=black] (5.,2.) -- (7.,2.) -- (7.,0.) -- (5.,0.) -- cycle;
\draw [shift={(-0.4442307692307692,1.)},line width=1.2pt]  plot[domain=-1.1527504411823672:1.1527504411823675,variable=\t]({1.*1.0942307692307691*cos(\t r)+0.*1.0942307692307691*sin(\t r)},{0.*1.0942307692307691*cos(\t r)+1.*1.0942307692307691*sin(\t r)});
\draw [shift={(1.4442307692307692,1.)},line width=1.2pt]  plot[domain=1.9888422124074259:4.29434309477216,variable=\t]({1.*1.0942307692307691*cos(\t r)+0.*1.0942307692307691*sin(\t r)},{0.*1.0942307692307691*cos(\t r)+1.*1.0942307692307691*sin(\t r)});
\draw [line width=1.6pt,<->] (2.3,1.)-- (4.3,1.);
\draw (2.1,2) node[anchor=north west] {$2$-move};
\draw [line width=1.2pt] (5.5,2.)-- (6.5,0.);
\draw [line width=1.2pt] (6.5,2.)-- (5.5,0.);
\end{tikzpicture} \qquad\qquad
\begin{tikzpicture}[x=15,y=15]
\fill[line width=.8pt,dash pattern=on 5pt off 5pt,fill=white,draw=black] (-0.5,2.) -- (1.5,2.) -- (1.5,0.) -- (-0.5,0.) -- cycle;
\fill[line width=.8pt,dash pattern=on 5pt off 5pt,fill=white,draw=black] (5.,2.) -- (7.,2.) -- (7.,0.) -- (5.,0.) -- cycle;
\draw [line width=1.2pt] (-.3,2.)-- (1.3,0.);
\draw [line width=1.2pt] (1.3,2.)-- (-.3,0.);
\draw [line width=1.2pt,rounded corners] (0.5,2.)-- (0,1.)-- (0.5,0.);
\draw [line width=1.6pt,<->] (2.3,1.)-- (4.3,1.);
\draw (2.1,2) node[anchor=north west] {$3$-move};
\draw [line width=1.2pt] (5.2,2.)-- (6.8,0.);
\draw [line width=1.2pt] (6.8,2.)-- (5.2,0.);
\draw [line width=1.2pt,rounded corners] (6,2.)-- (6.5,1.)-- (6,0.);
\begin{scope}[shift={(5.5,0)}]
\fill[line width=.8pt,dash pattern=on 5pt off 5pt,fill=white,draw=black] (5.,2.) -- (7.,2.) -- (7.,0.) -- (5.,0.) -- cycle;
\draw [line width=1.6pt,<->] (2.3,1.)-- (4.3,1.);
\draw (2.1,2) node[anchor=north west] {$3$-move};
\draw [line width=1.2pt] (5.2,2.)-- (6,0.);
\draw [line width=1.2pt] (6,2.)-- (6.8,0.);
\draw [line width=1.2pt] (6.8,2.)-- (5.2,0.);
\end{scope}
\end{tikzpicture}
\caption{Local $2$- and $3$-moves for $n$-diagrams}\label{P:Moves}
\end{figure}
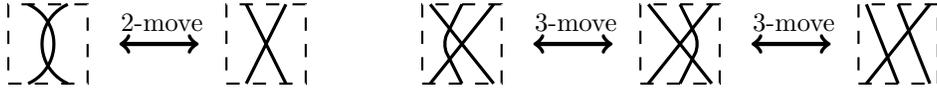

One \emph{composes} two $n$-diagrams by putting the first one on top of the second one and shrinking the result; see Fig.~\ref{P:Compo} (left). This defines a monoid structure, which survives in the quotient by $2$- and $3$-moves. This quotient monoid will be called the \emph{diagrammatic linear Hecke-Kiselman monoid}, denoted by $\DHKL$. Like braid monoids, it is generated by the \emph{elementary $n$-diagrams} $d_1,\ldots,d_n$ from Fig.~\ref{P:Compo} (right).  

\begin{figure}[h]
\centering
\begin{tikzpicture}[x=15,y=15]
\draw [line width=1.pt] (3.,0.)-- (4.,1.);
\draw [line width=1.pt] (2.,0.)-- (3.,1.);
\draw [line width=1.pt] (4.,0.)-- (2.,1.);
\draw [line width=1.pt] (4.,1.)-- (3.,2.);
\draw [line width=1.pt] (3.,1.)-- (4.,2.);
\draw [line width=1.pt] (2.,1.)-- (2.,2.);
\draw [line width=0.8pt,dash pattern=on 5pt off 5pt] (1.,2.)-- (5.,2.);
\draw [line width=0.8pt,dash pattern=on 5pt off 5pt] (1.,0.)-- (5.,0.);
\draw [line width=0.8pt,dash pattern=on 5pt off 5pt] (1.,1.)-- (5.,1.);
\draw (5.1,1.04) node[anchor=north west] {$d'$};
\draw (5.04,2.14) node[anchor=north west] {$d$};
\draw (-1.2,1.68) node[anchor=north west] {$dd'=$};
\begin{scriptsize}
\draw [fill=uuuuuu] (2.,0.) circle (1.0pt);
\draw [fill=uuuuuu] (3.,0.) circle (1.0pt);
\draw [fill=uuuuuu] (4.,0.) circle (1.0pt);
\draw [fill=uuuuuu] (2.,1.) circle (1.0pt);
\draw [fill=uuuuuu] (3.,1.) circle (1.0pt);
\draw [fill=uuuuuu] (4.,1.) circle (1.0pt);
\draw [fill=uuuuuu] (2.,2.) circle (1.0pt);
\draw [fill=uuuuuu] (3.,2.) circle (1.0pt);
\draw [fill=uuuuuu] (4.,2.) circle (1.0pt);
\end{scriptsize}
\end{tikzpicture} \qquad\qquad\qquad
\begin{tikzpicture}[x=15,y=15]
\draw [line width=1.pt] (1.,0.)-- (1.,1.);
\draw [line width=1.pt] (2.,0.)-- (3.,1.);
\draw [line width=1.pt] (3.,0.)-- (2.,1.);
\draw [line width=1.pt] (4.,0.)-- (4.,1.);
\draw (-1.2,1) node[anchor=north west] {$d_2=$};
\begin{scriptsize}
\draw [fill=uuuuuu] (1.,0.) circle (1.0pt);
\draw [fill=uuuuuu] (2.,0.) circle (1.0pt);
\draw [fill=uuuuuu] (3.,0.) circle (1.0pt);
\draw [fill=uuuuuu] (4.,0.) circle (1.0pt);
\draw [fill=uuuuuu] (1.,1.) circle (1.0pt);
\draw [fill=uuuuuu] (2.,1.) circle (1.0pt);
\draw [fill=uuuuuu] (3.,1.) circle (1.0pt);
\draw [fill=uuuuuu] (4.,1.) circle (1.0pt);
\draw[color=uuuuuu] (1,-0.31) node {1};
\draw[color=uuuuuu] (2,-0.31) node {2};
\draw[color=uuuuuu] (3,-0.31) node {3};
\draw[color=uuuuuu] (4,-0.31) node {4};
\end{scriptsize}
\end{tikzpicture}
\caption{Left: Composing $2$-diagrams $d$ and $d'$. Right: The elementary $3$-diagram $d_2$}\label{P:Compo}
\end{figure}
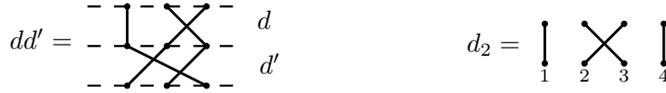

\begin{pro}\label{Pr:HKvsDHK}
The monoid map $\varepsilon_n \colon \HKL \to \DHKL$ sending each Hecke-Kiselman generator $x_i$ to the elementary $n$-diagram $d_i$ is bijective.
\end{pro}

\begin{proof}
The map $\varepsilon_n$ is well defined since the defining relations \eqref{E:idempot}, \eqref{E:FarComm}, and \eqref{E:braid} of $\HKL$ are realised in $\DHKL$ by $2$-moves, isotopies, and $3$-moves respectively. Moreover, $\varepsilon_n$ is surjective since the $d_i$s generate the monoid $\DHKL$. 

Next, we need to check the injectivity of $\varepsilon_n$. The relation $\varepsilon_n(w)=\varepsilon_n(w')$ means that the $n$-diagrams $\varepsilon_n(w)$ and $\varepsilon_n(w')$ are related by a sequence of isotopies and $2$- and $3$-moves. An isotopy can be realised in $\HKL$ by the far-commutativity relations \eqref{E:FarComm}. Now, given a $2$- or $3$-move, one can apply an isotopy on both sides of the move so that no double points have heights lying between the heights of the double points involved in the move. And such ``close'' $2$- and $3$-moves are realised in $\HKL$ by the relations \eqref{E:idempot} and \eqref{E:braid} respectively.
\end{proof}

Let us now move to the $\HKC$. By an \emph{extended $\tilde{n}$-diagram} we will mean a continuous map $d \colon \bigsqcup_{i=1}^{n} I_i \to \SSS^1 \times I$, where the circle $\SSS^1$ is seen as the interval $[1,n+1]$ with glued endpoints $1$ and $n+1$, satisfying the conditions \ref{it:first}-\ref{it:last} above, and considered up to isotopy. They can be thought of as braid or $4$-valent graph diagrams on a cylinder. Some examples are given in Fig.~\ref{P:CylDiagsEx}. Here cylinders are cut along the line $x=1$ and represented by squares. Braids on a cylinder have been extensively studied in the literature, and are known under various names (annular, affine etc.); see for example \cite{cyl1,cyl2,cyl3} and references thereto.
\begin{figure}[h]
\centering
\begin{tikzpicture}[x=15,y=15]
\draw [line width=1.pt,color=uuuuuu] (0.,0.)-- (2.,4.);
\draw [line width=1.pt] (3.,0.)-- (1.,4.);
\draw [line width=1.pt,color=uuuuuu] (1.,0.)-- (0.,2.);
\draw [line width=1.pt] (4.,2.)-- (3.,4.);
\draw [line width=1.pt] (2.,0.)-- (4.,4.);
\draw [line width=0.8pt,dash pattern=on 5pt off 5pt] (0.,4.)-- (4.,4.);
\draw [line width=0.8pt,dash pattern=on 5pt off 5pt] (0.,0.)-- (4.,0.);
\draw [line width=1.pt,dotted] (0.,4.)-- (0.,0.);
\draw [line width=1.pt,dotted] (4.,4.)-- (4.,0.);
\begin{scriptsize}
\draw [fill=uuuuuu] (0.,0.) circle (1.0pt);
\draw [fill=uuuuuu] (1.,0.) circle (1.0pt);
\draw [fill=uuuuuu] (2.,0.) circle (1.0pt);
\draw [fill=uuuuuu] (3.,0.) circle (1.0pt);
\draw [fill=uuuuuu] (4.,0.) circle (1.0pt);
\draw [fill=uuuuuu] (0.,4.) circle (1.0pt);
\draw [fill=uuuuuu] (1.,4.) circle (1.0pt);
\draw [fill=uuuuuu] (2.,4.) circle (1.0pt);
\draw [fill=uuuuuu] (3.,4.) circle (1.0pt);
\draw [fill=uuuuuu] (4.,4.) circle (1.0pt);
\draw[color=uuuuuu] (0,-0.31) node {1};
\draw[color=uuuuuu] (1,-0.31) node {2};
\draw[color=uuuuuu] (2,-0.31) node {3};
\draw[color=uuuuuu] (3,-0.31) node {4};
\draw[color=uuuuuu] (4,-0.31) node {5};
\end{scriptsize}
\end{tikzpicture}\qquad\qquad\qquad\qquad
\begin{tikzpicture}[x=15,y=15]
\draw [line width=1.pt,color=uuuuuu] (0.,0.)-- (1.,4.);
\draw [line width=1.pt,color=uuuuuu] (1.,0.)-- (2.,4.);
\draw [line width=1.pt,color=uuuuuu] (2.,0.)-- (3.,4.);
\draw [line width=1.pt,color=uuuuuu] (3.,0.)-- (4.,4.);
\draw [line width=0.8pt,dash pattern=on 5pt off 5pt] (0.,4.)-- (4.,4.);
\draw [line width=0.8pt,dash pattern=on 5pt off 5pt] (0.,0.)-- (4.,0.);
\draw [line width=1.pt,dotted] (0.,4.)-- (0.,0.);
\draw [line width=1.pt,dotted] (4.,4.)-- (4.,0.);
\begin{scriptsize}
\draw [fill=uuuuuu] (0.,0.) circle (1.0pt);
\draw [fill=uuuuuu] (1.,0.) circle (1.0pt);
\draw [fill=uuuuuu] (2.,0.) circle (1.0pt);
\draw [fill=uuuuuu] (3.,0.) circle (1.0pt);
\draw [fill=uuuuuu] (4.,0.) circle (1.0pt);
\draw [fill=uuuuuu] (0.,4.) circle (1.0pt);
\draw [fill=uuuuuu] (1.,4.) circle (1.0pt);
\draw [fill=uuuuuu] (2.,4.) circle (1.0pt);
\draw [fill=uuuuuu] (3.,4.) circle (1.0pt);
\draw [fill=uuuuuu] (4.,4.) circle (1.0pt);
\draw[color=uuuuuu] (0,-0.31) node {1};
\draw[color=uuuuuu] (1,-0.31) node {2};
\draw[color=uuuuuu] (2,-0.31) node {3};
\draw[color=uuuuuu] (3,-0.31) node {4};
\draw[color=uuuuuu] (4,-0.31) node {5};
\end{scriptsize}
\end{tikzpicture}
\caption{Left: A $\tilde{4}$-diagram. Right: The extended $\tilde{4}$-diagram $t$}\label{P:CylDiagsEx}
\end{figure}
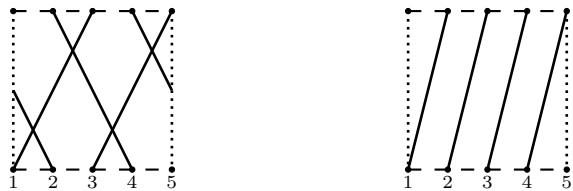

Extended $\tilde{n}$-diagrams form a monoid under the stack-and-shrink composition. Consider the submonoid generated by the \emph{elementary $\tilde{n}$-diagrams} $\tilde{d}_1,\ldots,\tilde{d}_n$, defined by the same pictures as the elementary $n$-diagrams $d_i$. Its elements will be called \emph{$\tilde{n}$-diagrams}. In other words, we forbid the $\frac{2\pi}{n}$ twist $t$ from Fig.~\ref{P:CylDiagsEx} (which, together with the $\tilde{d}_i$s, generates the whole monoid of extended $\tilde{n}$-diagrams). For example, the left $\tilde{4}$-diagram from Fig.~\ref{P:CylDiagsEx} decomposes as $\tilde{d}_2\tilde{d}_4\tilde{d}_1\tilde{d}_3$; actually, it has $4$ different decompositions due to isotopy: one can exchange $\tilde{d}_1$ and $\tilde{d}_3$, and $\tilde{d}_2$ and $\tilde{d}_4$. The quotient of this monoid by the same local $2$- and $3$-relations as in the linear case will be called the \emph{diagrammatic circular Hecke-Kiselman monoid} $\DHKC$.

The proof of Proposition \ref{Pr:HKvsDHK} extends verbatim to this new setting, and yields
\begin{pro}\label{Pr:HKvsDHK_C}
The monoid map $\tilde{\varepsilon}_n \colon \HKC \to \DHKC$ sending each Hecke-Kiselman generator $x_i$ to the elementary $\tilde{n}$-diagram $\tilde{d}_i$ is bijective.
\end{pro}

\section{Weakly entangled braid diagrams}

In this section we show how to ``disentangle'' any $n$- or $\tilde{n}$-diagram using $2$- and $3$-moves. We will later see that the result of this disentanglement is unique.

A \emph{weakly entangled braid $n$-diagram}, or \emph{$n$-web} for short, is an $n$-diagram without bigons nor triangles. That is, any two strands intersect at most once, and there are no pairwise intersecting strand triples. The set of $n$-webs, considered up to isotopy as usual, is denoted by $\WHKL$. Note that a composition of webs need not be one. Webs have two useful alternative descriptions:

\begin{pro}\label{Pr:Webs}
For an $n$-diagram $d$, the following statements are equivalent:
\begin{enumerate}
\item\label{it1} $d$ is an $n$-web;
\item\label{it2} $d$ has neither minimal bigons nor minimal triangles;
\item\label{it3} $d$ can be isotoped to an $n$-diagram where each strand is either vertical or projects injectively to the $x$-axis, and when two strands cross, the $x$-coordinate is strictly increasing (when followed from bottom to top) for one of them and strictly decreasing for the second one.
\end{enumerate}
\end{pro}

Here a bigon/triangle is called \emph{minimal} if it is not intersected by other strands. For instance, the first $4$-diagram from Fig.~\ref{P:DiagsEx} is a $4$-web; the second one has a minimal bigon and a non-minimal triangle; and the third one has a non-minimal bigon and both minimal and non-minimal triangles.

The condition on strands in \ref{it3} is equivalent to saying that, when followed from bottom to top, they always go straight up / to the right / to the left. Such strands will be called \emph{trivial}, \emph{right}, and \emph{left} respectively. Moreover, each crossing should be an intersection of a right strand and a left strand. For instance, in the first $4$-diagram from Fig.~\ref{P:DiagsEx}, strands 1 and 3 are right, 2 and 4 are left, and 5 is trivial. Here and below strands are numbered by the $x$-coordinate of their lower endpoint. The trivial/right/left property does not depend on the concrete diagram representing an $n$-web, since it is determined by the endpoints of the strand only. It is thus legitimate to talk about a trivial/right/left strand of an $n$-web. 

\begin{proof}
Implication \ref{it1} $\,\Rightarrow\,$ \ref{it2} is trivial.

To show \ref{it3} $\,\Rightarrow\,$ \ref{it1}, present an $n$-diagram as explained in \ref{it3}. In a bigon, the strand going to the left at the lower intersection goes to the right at the upper one. The same happens to the strand connecting the lower and the upper intersections in a triangle. Thus an $n$-diagram satisfying \ref{it3} has neither bigons nor triangles. 

To show \ref{it2} $\,\Rightarrow\,$ \ref{it3}, take an $n$-diagram without minimal bigons nor minimal triangles, and present it as a composition of elementary $n$-diagrams. Let us follow each strand from bottom to top. If each strand goes in the same direction (left or right) at each intersection it crosses, then, by slightly moving some vertical segments if necessary, one gets an $n$-diagram satisfying \ref{it3}. Now, assume that one strand goes to different directions at some of its intersections. Then among pairs of intersections connected by a common strand changing direction in between, choose one with the minimal height difference. Let us denote this common strand by $s$, the two intersections realising the desired property by $A$ and $B$, and the two other strands going through $A$ and $B$ by $s_A$ and $s_B$ respectively; see Fig.~\ref{P:WebProof}. 
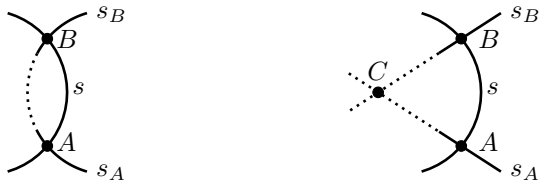
\begin{figure}[h]
\centering
\begin{tikzpicture}[x=30,y=30]
\draw [shift={(-0.2916666666666667,1.)},line width=1.pt]  plot[domain=-1.2870022175865685:1.2870022175865687,variable=\t]({1.*1.0416666666666667*cos(\t r)+0.*1.0416666666666667*sin(\t r)},{0.*1.0416666666666667*cos(\t r)+1.*1.0416666666666667*sin(\t r)});
\draw [shift={(1.2916666666666663,1.)},line width=1.pt,dotted]  plot[domain=2.649899594646554:3.6144337800879773,variable=\t]({1.*1.0416666666666663*cos(\t r)+0.*1.0416666666666663*sin(\t r)},{0.*1.0416666666666663*cos(\t r)+1.*1.0416666666666663*sin(\t r)});
\draw [shift={(1.2916666666666639,1.)},line width=1.pt]  plot[domain=1.8545904360032228:2.649899594646555,variable=\t]({1.*1.041666666666663*cos(\t r)+0.*1.041666666666663*sin(\t r)},{0.*1.041666666666663*cos(\t r)+1.*1.041666666666663*sin(\t r)});
\draw [shift={(1.2916666666666656,1.)},line width=1.pt]  plot[domain=3.6144337800879764:4.4285948711763625,variable=\t]({1.*1.041666666666665*cos(\t r)+0.*1.041666666666665*sin(\t r)},{0.*1.041666666666665*cos(\t r)+1.*1.041666666666665*sin(\t r)});
\draw (0.7,1.24) node[anchor=north west] {$s$};
\draw (0.5,0.6) node[anchor=north west] {$A$};
\draw (0.5,1.9) node[anchor=north west] {$B$};
\draw (1,2.2) node[anchor=north west] {$s_B$};
\draw (1,0.2) node[anchor=north west] {$s_A$};
\draw [fill=uuuuuu] (0.5,1.6770032003863304) circle (2pt);
\draw [fill=uuuuuu] (0.5,0.3229967996136698) circle (2pt);
\end{tikzpicture}\qquad\qquad\qquad\qquad
\begin{tikzpicture}[x=30,y=30]
\draw [shift={(-0.2916666666666667,1.)},line width=1.pt]  plot[domain=-1.2870022175865685:1.2870022175865687,variable=\t]({1.*1.0416666666666667*cos(\t r)+0.*1.0416666666666667*sin(\t r)},{0.*1.0416666666666667*cos(\t r)+1.*1.0416666666666667*sin(\t r)});
\draw [line width=1.pt,dotted] (0.21892207381119294,1.4954286591842665)-- (-0.9182918309167227,0.7607957557777041);
\draw [line width=1.pt,dotted] (-0.9182918309167223,1.2392042442222966)-- (0.2207533216825331,0.5033883664122494);
\draw [line width=1.pt] (1.,2.)-- (0.21892207381119294,1.4954286591842665);
\draw [line width=1.pt] (1.,0.)-- (0.2207533216825331,0.5033883664122494);
\draw (0.7,1.24) node[anchor=north west] {$s$};
\draw (0.6,0.6) node[anchor=north west] {$A$};
\draw (0.6,1.9) node[anchor=north west] {$B$};
\draw (1,2.2) node[anchor=north west] {$s_B$};
\draw (1,0.2) node[anchor=north west] {$s_A$};
\draw (-0.8,1.5) node[anchor=north west] {$C$};
\draw [fill=uuuuuu] (0.5,1.6770032003863304) circle (2pt);
\draw [fill=uuuuuu] (0.5,0.3229967996136698) circle (2pt);
\draw [fill=uuuuuu] (-0.5480029542027676,1.) circle (2pt);
\end{tikzpicture}
\caption{The closest intersections connected by a direction-changing strand}\label{P:WebProof}
\end{figure}
By the minimality condition in the choice of $A$ and $B$, $s$ does not run through any intersections between $A$ and $B$. If $s_A=s_B$, then, since $s_A$ cannot cross $s$ between $A$ and $B$, $s_A$ and $s$ form a bigon between $A$ and $B$. This bigon is minimal: indeed, if a strand intersects this bigon, then it crosses $s_A$ at some point $D$ between $A$ and $B$, and the pair $(D,A)$ or $(D,B)$ breaks the minimality condition in the choice of the pair $(A,B)$. In the alternative case $s_A \neq s_B$, the strands $s_A$ and $s_B$ have to cross at some point $C$ between $A$ and $B$. This yields a triangle $ABC$, which is minimal by an argument exploring once again the minimality condition on $A$ and $B$. 
\end{proof}

We will now show that $2$- and $3$-moves can reduce any $n$-diagram to an $n$-web, whose uniqueness  will be proved later on. In other words, $n$-webs yield normal forms of $n$-diagrams, easy both to construct and to compare.

Concretely, the inclusion of the set of $n$-webs into the set of $n$-diagrams induces a map $\iota \colon \WHKL \to \DHKL$.

\begin{pro}\label{Pr:Disentangle}
The map $\iota \colon \WHKL \to \DHKL$ is surjective.
\end{pro}

\begin{proof}
Take an $n$-diagram $d$ which is not an $n$-web. By Proposition~\ref{Pr:Webs}, $d$ contains a minimal bigon/triangle. A $2$-/$3$-move can be used to remove it, reducing the crossing number of $d$ by $1$. Since the crossing number can be reduced only a finite number of times, one can iterate this process until getting an $n$-web, which is equivalent to $d$ modulo $2$- and $3$-moves. Thus $d$ lies in the image of~$\iota$.
\end{proof}

Let us now turn to $\tilde{n}$-diagrams. They can have two types of bigons/triangles: \emph{contractible} and \emph{non-contractible}, according to whether or not the  bigon/triangle delimits a contractible part of the cylinder; see Fig.~\ref{P:Contractible}.

\begin{figure}[h]
\centering
\begin{tikzpicture}[x=20,y=20]
\draw[fill=mygrey,line width=2pt]
(3,1) 
-- plot[domain=-.7:.7,variable=\t]({cos(\t r)+2.87},{sin(\t r)+1})
-- cycle;
\draw[fill=mygrey,line width=2pt] plot[domain=-.6:.6,variable=\t]({cos(\t r)-0.29},{sin(\t r)+1})
--plot[domain=2.54:3.76,variable=\t]({cos(\t r)+1.29},{sin(\t r)+1});
\draw [shift={(-0.29,1.)},line width=1pt]  plot[domain=-1.29:1.29,variable=\t]({cos(\t r)},{sin(\t r)});
\draw [shift={(1.29,1.)},line width=1pt]  plot[domain=1.85:4.43,variable=\t]({cos(\t r)},{sin(\t r)});
\draw [line width=1pt] (2.,2.)-- (4.,0.);
\draw [line width=1pt] (2.,0.)-- (4.,2.);
\draw [shift={(2.87,1.)},line width=1pt]  plot[domain=-1.44:1.44,variable=\t]({cos(\t r)},{sin(\t r)});
\end{tikzpicture}\qquad\qquad\qquad\qquad
\begin{tikzpicture}[x=20,y=20]
\draw [line width=1pt] (1.,0.)-- (2.,3.);
\draw [line width=1pt] (2.,0.)-- (4.,3.);
\draw [line width=1pt] (3.,0.)-- (1.,2.);
\draw [line width=1pt] (4.,2.)-- (3.,3.);
\draw [line width=2.pt] (2.4,0.6)-- (3.6,2.4);
\draw [line width=2.pt] (2.4,0.6)-- (1.,2.);
\draw [line width=2.pt] (4.,2.)-- (3.6,2.4);
\draw [line width=0.8pt,dotted] (1.,0.)-- (1.,3.);
\draw [line width=0.8pt,dotted] (4.,0.)-- (4.,3.);
\end{tikzpicture}\qquad\qquad
\begin{tikzpicture}[x=20,y=20]
\draw [line width=0.8pt,dotted] (1.,0.)-- (1.,3.);
\draw [line width=0.8pt,dotted] (4.,0.)-- (4.,3.);
\draw [line width=1pt] (2.,0.)-- (4.,2.);
\draw [line width=1pt] (1.,2.)-- (2.,3.);
\draw [line width=1pt] (1.,0.)-- (1.,1.);
\draw [line width=1pt] (4.,1.)-- (1.,3.);
\draw [line width=1pt] (3.,3.)-- (3.,0.);
\draw [line width=2.pt] (3.,1.6666666666666667)-- (1.6,2.6);
\draw [line width=2.pt] (1.6,2.6)-- (1.,2.);
\draw [line width=2.pt] (4.,2.)-- (3.,1.);
\draw [line width=2.pt] (3.,1.)-- (3.,1.6666666666666667);
\end{tikzpicture}
\caption{Contractible (left figures) and non-contractible (right figures) bigons and triangles}\label{P:Contractible}
\end{figure}
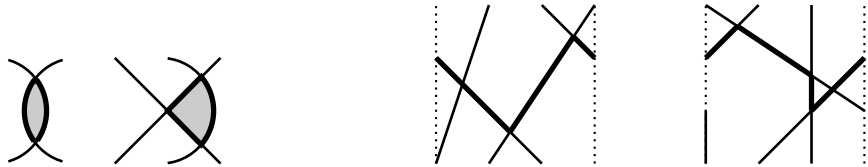

All arguments from this section can be adapted verbatim to $\tilde{n}$-diagrams if one considers only contractible bigons and triangles, and if by ``right/left'' one means ``going in the positive/negative direction of the $\SSS^1$ factor of the cylinder $\SSS^1 \times I$''. One obtains the set $\WHKC$ of \emph{$\tilde{n}$-webs}, and a map $\tilde{\iota} \colon \WHKC \to \DHKC$ satisfying:

\begin{pro}\label{Pr:Disentangle_C}
The map $\tilde{\iota} \colon \WHKC \to \DHKC$ is surjective.
\end{pro}

\begin{rem}
It is interesting to note that a diagram with a non-contractible triangle necessarily contains either a contractible triangle or a contractible bigon. This can be seen using a minimality argument similar to that from the proof of Proposition~\ref{Pr:Webs}. Thus an $\tilde{n}$-web cannot contain non-contractible triangles. Another way of seeing this is observing that two of the three strands forming a triangle (contractible or not) in an $\tilde{n}$-web should be of the same type, right or left, and thus cannot intersect.
\end{rem}

\begin{rem}
The disentanglement procedures for diagrams from Propositions \ref{Pr:Disentangle} and \ref{Pr:Disentangle_C} can be translated into rewriting procedures for words in $\HKL$ and $\HKC$. Indeed, killing a bigon or a triangle boils down to applying a generalised version of relation \eqref{E:idempot} or \eqref{E:braid} respectively in the length-reducing direction. The generalisation in question allows one to insert some generators $x_j$ coherently on both sides, provided that these generators far-commute, in the sense of \eqref{E:FarComm}, with the generators involved in the relation. This rewriting approach is developed in \cite{MeOk,ArDA19}.
\end{rem}

\section{Integer sequences and permutations encoding webs}

An $n$-diagram $d$ induces a permutation $p(d) \in S_{n+1}$ of $\{1,2,\ldots,n+1\}$: just send the $x$-coordinate of the lower endpoint of each strand to the $x$-coordinate of its upper endpoint. In this section we show that an $n$-web $w$ is completely determined by the permutation $p(w)$, which establishes an explicit bijection between $\WHKL$ and $321$-avoiding permutations from $S_{n+1}$. Moreover, all one needs to reconstruct $p(w)$ are the endpoints of all right strands of $w$, encoded with two integer sequences. Similar two-integer-sequence description is developed for $\tilde{n}$-webs. This yields a very efficient way of representing, comparing, and enumerating $n$- and $\tilde{n}$-webs.

Denote by $\Perm$ the set of all \emph{$321$-avoiding permutations} from $S_{n+1}$, that is, permutations $s$ that do not completely permute any triple: one cannot simultaneously have $i<j<k$ and $s(i)>s(j)>s(k)$. Also, denote by $\II$ the set of all \emph{increasing couples of increasing integer sequences} bounded by $1$ and $n+1$, that is, $2k$ integers, for any $0 \leqslant k \leqslant n$, satisfying the inequalities
\begin{center}
\begin{tabular}{*{11}{>$c<$}}
 &  & b_1 & < & b_2 & < & \ldots & < & b_k & \leqslant & n+1 \\
 &  & \vee &  & \vee &  & \ldots &  & \vee & & \\
 1 & \leqslant & a_1 & < & a_2 & < & \ldots & < & a_k & & 
\end{tabular}
\end{center}

For a permutation $s \in S_{n+1}$, denote by $r(s)$ its \emph{right sequence couple}, consisting of the ordered sequence $(a_t)$ of all the $i$s with $s(i)>i$, and of the sequence $(b_t)$ of the same size defined by $b_t=s(a_t)$. This sequence couple satisfies all conditions from the definition of $\II$, except for the monotonicity of $(b_t)$.  

\begin{pro}\label{Pr:SeqPerm}
The maps $p$ and $r$ above induce bijections
\[\WHKL \underset{1:1}{\overset{\pi}{\longrightarrow}} \Perm \underset{1:1}{\overset{\rho}{\longrightarrow}} \II.\]
\end{pro}

This proposition allows us to talk about the \emph{right sequence couple of an $n$-web $w$}, which is simply $\rho \circ \pi (w)$. For instance, the right sequence couple of the first $4$-diagram from Fig.~\ref{P:DiagsEx} is $((1,3),(3,4))$.

\begin{proof}
Take an $n$-web $w$. Its strands $i < j$ intersect if and only if $p(i)>p(j)$. Since $w$ has no pairwise intersecting triples, $p(w)$ is $321$-avoiding. Hence the map $\pi$ is well defined.

We will next construct the \emph{drawing map} $\delta \colon \Perm \to \WHKL$, and show that it is the inverse of $\pi$. Take a permutation $s \in \Perm$, and construct a diagram in $\RR \times I$ by connecting with a straight segment the points $(i,0)$ and $(s(i),1)$ for all $i \in \{1,2,\ldots,n+1\}$. These segments will have only simple intersections, since $s$ is $321$-avoiding. So we get an $n$-diagram. Two segments cannot form a bigon, and triangles are ruled out as $s$ is $321$-avoiding. Thus we get an $n$-web, denoted by $\delta(s)$. By construction, $\pi \circ \delta (s)=s$. Next, using an inductive argument, we will explain how to isotope any $n$-web $w$ to $\delta \circ \pi (w)$. This can be thought of as the linearisation of $w$. We will actually prove the statement for a slightly wider class of objects, called \emph{wide $n$-webs}: their upper endpoints are allowed to be of the form $(i,1)$ for any integer $i$. The case $n=1$ is trivial. Let us then consider any $n$, and assume the case $n-1$ settled. Proposition~\ref{Pr:Webs} easily adapts to wide $n$-webs: the relevant definition for trivial/right/left strands is crossingless or going to the right/left of any intersecting strand, respectively. Thus all strands of $w$ can be assumed to be trivial, right, or left. Consider the strand $l$ in $w$ connecting the point $(1,0)$ to some $(i,1)$. It is either trivial or right. Then $l$ crosses either no strands or several left strands, called \emph{active}. The strand $l$ can be isotoped up, as in Fig.~\ref{P:Linearisation}.A; indeed, the grey zone contains only left strands, which do not intersect. The remaining strands are static during this isotopy. In the pictures in Fig.~\ref{P:Linearisation}, only $l$ and the active strands are drawn, and the old and the new positions of $l$ are depicted using a solid and a dashed line respectively. If $l$ is isotoped close enough to the upper border $\RR \times \{1\}$, then the induction hypothesis can be applied to the remaining $n$ strands, while keeping $l$ static. In the resulting $n$-web, $l$ can be isotoped to the straight position, as in Fig.~\ref{P:Linearisation}.C, since once again the grey zone cannot contain any intersections. The resulting $n$-web is precisely $\delta \circ \pi (w)$.

\begin{figure}[h]
\centering
\begin{tikzpicture}[x=20,y=20]
\fill[line width=0.4pt,color=mygrey,fill=mygrey,rounded corners] (1.,0.) -- (2.1,0.8) -- (2.42,1.85) -- (3.75,2.17) -- (4.,3.) -- (1.07,2.4) -- cycle;
\draw [line width=1.pt,rounded corners] (1.,0.) -- (2.1,0.8)-- (2.42,1.85)-- (3.75,2.17)-- (4.,3.);
\draw [line width=1.pt,dash pattern=on 3pt off 3pt,rounded corners] (4.,3.)-- (1.07,2.4)-- (1.,0.);
\draw [line width=1.pt,rounded corners] (1.,3.)-- (1.7,1.1)-- (2.,0.);
\draw [line width=1.pt,rounded corners] (2.,3.)-- (1.95,1.72)-- (2.8,0.7)-- (3.,0.);
\draw [line width=1.pt,rounded corners] (3.,3.)-- (3.9,1.3)-- (4.,0.);
\draw (4.5,1.5) node {$A$};
\end{tikzpicture}\qquad\qquad
\begin{tikzpicture}[x=20,y=20]
\draw [line width=1.pt,rounded corners] (4.,3.)-- (1.07,2.4)-- (1.,0.);
\draw [line width=1.pt] (1.,3.)-- (2.,0.);
\draw [line width=1.pt] (2.,3.)-- (3.,0.);
\draw [line width=1.pt] (3.,3.)-- (4.,0.);
\draw (4.5,1.5) node {$B$};
\end{tikzpicture}\qquad\qquad
\begin{tikzpicture}[x=20,y=20]
\fill[line width=0.4pt,color=mygrey,fill=mygrey,rounded corners] (4.,3.)-- (1.07,2.4)-- (1.,0.) -- cycle;
\draw [line width=1.pt,rounded corners] (4.,3.)-- (1.07,2.4)-- (1.,0.);
\draw [line width=1.pt,dash pattern=on 3pt off 3pt] (4.,3.)-- (1.,0.);
\draw [line width=1.pt] (1.,3.)-- (2.,0.);
\draw [line width=1.pt] (2.,3.)-- (3.,0.);
\draw [line width=1.pt] (3.,3.)-- (4.,0.);
\draw (4.5,1.5) node {$C$};
\end{tikzpicture}
\caption{Linearising an $n$-web inductively}\label{P:Linearisation}
\end{figure}
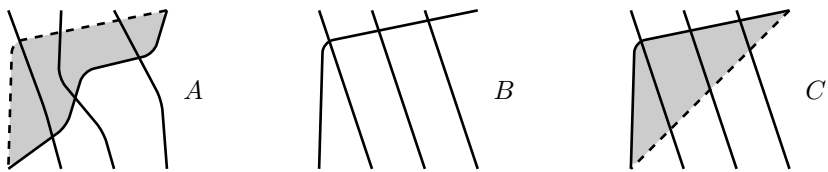

Let us now turn to sequence couples. To show that the map $\rho$ is well defined, we need to take a $321$-avoiding permutation $s$ and show that $r(s) \in \II$. That is, for any $i<j$ such as $i<s(i)$ and $j<s(j)$, we need to check that $s(i)<s(j)$. Since $j<s(j)$, there are more elements from $\{1,2,\ldots,n+1\}$ that are $>j$ than those $>s(j)$. As a result, there exists a $k>j$ with $s(k)<s(j)$. But then an $i$ with $i<j$ and $s(i)>s(j)$ would create a forbidden $321$ pattern in $s$.

It remains to describe the \emph{completion map} $\kappa \colon \II \to \Perm$, and show that it is the inverse of $\rho$. Take a sequence couple $c=((a_1,\ldots,a_k),\ (b_1,\ldots,b_k))$ from $\II$. Let $a_{k+1}<\ldots<a_{n+1}$  and $b_{k+1}<\ldots<b_{n+1}$ be the complements of the sequences $(a_t)$ and, respectively, $(b_t)$ in $\{1,2,\ldots,n+1\}$. Define a permutation $\kappa(c) \in S_{n+1}$ by sending each $a_t$ to $b_t$. One needs to show that $\kappa(c)$ is $321$-avoiding. Take any $i<j<m$. Two of the three elements $i,j,k$ have to belong either to $\{a_1,\ldots,a_k\}$ or to $\{a_{k+1},\ldots,a_{n+1}\}$. But then they are not reversed by $\kappa(c)$, forbidding the $321$ pattern. Hence a well defined map $\kappa \colon \II \to \Perm$. Moreover, with the above notations, we will prove that $\rho \circ \kappa (c)=c$. That is, the permutation $\kappa (c)$ should not increase any of $a_{k+1}, \ldots, a_{n+1}$. Suppose that, on the contrary, $a_t<s(a_t)$ for some $t>k$. As in the previous paragraph, this implies that $b_u=s(a_u)<s(a_t)=b_t$ for some $a_u>a_t$. One cannot have $u>k$, since then one would have $b_u>b_t$ by the construction of $\kappa(c)$. But $u \leqslant k$ is not possible either, since this would imply $a_u < b_u$, and then in $\delta \circ \kappa(c)$ two right strands $a_u$ and $a_t$ would intersect. Finally, for any $s \in \Perm$, one has $\kappa \circ \rho (s)=s$. To see this, it suffices to check that $\kappa$ completes the permutation $s$ correctly on the elements non-increased by~$s$. That is, for $i<j$ with $s(i) \leqslant i, s(j) \leqslant j$, one necessarily has $s(i)<s(j)$. The argument here is the same as that in the proof of $r(s) \in \II$ in the previous paragraph. 
\end{proof}

Observe that all the maps from the proposition, as well as their inverses, are explicit, and have  realisations linear in $n$ (for $\rho$ and its inverse) or in the diagram's crossing number (for $\pi$ and its inverse). 

\begin{rem}
Diminishing each $b_i$ by $1$, one gets a slightly modified condition on the sequence couple: 
\begin{center}
\begin{tabular}{*{11}{>$c<$}}
 &  & b_1 & < & b_2 & < & \ldots & < & b_k & \leqslant & n \\
 &  & \geqrot &  & \geqrot &  & \geqrot &  & \geqrot & & \\
 1 & \leqslant & a_1 & < & a_2 & < & \ldots & < & a_k & & 
\end{tabular}
\end{center}
This form suggests a bijection between $\II$ and monotonic lattice paths along the edges of a grid with $(n+1) \times (n+1)$ square cells which do not pass above the diagonal: just record the coordinates of the ``corners'' of such a path. These paths provide one of the most classical realisations of the \emph{Catalan number} $C_{n+1}$. Recording all integer points of such a path, one gets a bijection with the \emph{Catalan monoid} $CM_{n+1}$, that is, the monoid of all order-preserving and order-decreasing total transformations of the set $\{1, 2, \ldots, n+1\}$.
\end{rem}

It is time to turn to $\tilde{n}$-webs. Here again the key information is contained in the endpoints of all the right strands. To completely determine an $\tilde{n}$-web, one needs to add certain homology information---namely, how many times each right strand turns around the core of the cylinder. These two types of information can be conveniently blended into something as simple as a couple of sequences, as follows. 

For an $\tilde{n}$-web $w$, denote by $r(w)$ its \emph{right sequence couple}. The first sequence in this couple is the ordered sequence $(a_t)$ of the $x$-coordinates of the lower endpoints of all the right strands of $w$. (Recall that $\tilde{n}$-webs are drawn on the cylinder $\SSS^1 \times I$, where the circle $\SSS^1$ is seen as the interval $[1,n+1]$ with glued endpoints $1$ and $n+1$; for a strand starting at $(1,0)\equiv(n+1,0)$, we choose $a_t=1$.) Next, define a sequence $(b_t)$ of the same size by $b_t=s(a_t)+n*h_t$, where $(s(a_t),1)$ is the upper endpoint of the strand starting at $(a_t,0)$ (again, when the $x$-coordinate is $1 \equiv n+1$, we choose $1$), and $h_t$ is the number of times this strand crosses the segment $\{n+\frac{1}{2}\} \times I$. The number $h_t$ is well defined if one works only with the representatives of $w$ where our strand is right, which exist by the $\tilde{n}$-analog of Proposition~\ref{Pr:Webs}.  

Denote by $\IIC$ the set of all \emph{$n$-close increasing couples of increasing integer sequences}, that is, $2k$ integers, for any $0 \leqslant k < n$, satisfying the inequalities
\begin{center}
\begin{tabular}{*{11}{>$c<$}}
 &  & b_1 & < & b_2 & < & \ldots & < & b_k & < & b_1+n \\
 &  & \vee &  & \vee &  & \ldots &  & \vee & & \\
 1 & \leqslant & a_1 & < & a_2 & < & \ldots & < & a_k & \leqslant & n
\end{tabular}
\end{center}

\begin{pro}\label{Pr:SeqPerm_C}
The map $r$ above induces a bijection
\[\WHKC \underset{1:1}{\overset{\tilde{\varrho}}{\longrightarrow}} \IIC.\]
\end{pro}

Note that we do not give an intermediate step for this bijection, which would play the role the $321$-avoiding permutations played for $n$-webs. This yields a slightly shorter proof. The price to pay is the difficulty of constructing the inverse for $\tilde{\varrho}$: we do it inductively, which is convenient for turning this construction into an algorithm, but makes the resulting $\tilde{n}$-web less tractable. An alternative ``straight-line'' version of $\tilde{\varrho}^{-1}$, similar to the ``straight-line'' inverse of the map $\WHKL \underset{1:1}{\overset{\pi}{\longrightarrow}} \Perm$ described in the proof of Proposition~\ref{Pr:SeqPerm}, will be given below. Less algorithmic, it is aimed at a diagrammatic-thinking reader.

\begin{proof}
Take an $\tilde{n}$-web $w$. We first need to check that $r(w) \in \IIC$. As usual, we rewrite $w$ as a composition of elementary $n$-diagrams $d_i$ where at each intersection a right strand goes to the right and a left strand goes to the left. This is possible due to the $\tilde{n}$-analog of Proposition~\ref{Pr:Webs}. We then follow $w$ from bottom to top, and observe how the sequence $(b_t)$ changes after each $d_i$. One starts with $b_t=a_t$ for all $t$. If $i<n$, and the strand going to the right at the crossing $d_i$ starts at some $(a_u,0)$, then $s(a_u)$ increases by $1$ and $h_u$ remains constant, so $b_u$ increases by $1$. If $i=n$, then $s(a_u)$ changes from $n$ to $1$, but $h_u$ increases by $1$ since the strand crosses the line $x=n+\frac{1}{2}$, so the overall increase of $b_u$ is once again $1$. The remaining right strands do not move, hence the remaining $b_t$ stay constant. Since the $b_t$ can only increase, and since each right strand goes through at least one crossing $d_i$, at the end one has $b_t>a_t$ for all $t$. Further, we start with a strictly increasing integer sequence $b_t=a_t$, with $b_k=a_k\leq n <n+1 \leq a_1+n=b_1+n$, and at each step exactly one $b_u$ increases by $1$. If at some step the monotonicity of $(b_t)$ or the condition $b_k<b_1+n$ is broken, then when this happens for the first time one has $b_u=b_{u+1}$ for some $u$, or else $b_k=b_1+n$. Thus two of the $b_t$s coincide $\operatorname{mod} n$, which means $s(a_u)=s(a_v)$ for $u \neq v$. Hence two of the upper endpoints coincide, which is impossible. Therefore $r(w)$ satisfies all the conditions defining $\IIC$.

We will now define the \emph{spiral map} $\tilde{\sigma} \colon \IIC \to \WHKC$, and show that it is the inverse of $\tilde{\varrho}$. Take a sequence couple $((a_t), (b_t)) \in \IIC$. We will work inductively on the size $k$ of $(b_t)$, and in the case of the same size on the maximal term $b_k$. One sets $\tilde{\sigma}(\emptyset,\emptyset)$ to be the trivial $\tilde{n}$-web. In the general case, let $b_u,\ldots,b_k$ be the maximal tail of $(b_t)$ consisting of consecutive numbers. In particular, $b_{u-1}<b_u -1$ if $b_{u-1}$ exists. Then decrease each of $b_u,\ldots,b_k$ by $1$, and discard all elements $a_t$ and $b_t$ with $b_t=a_t$. Denote the new sequence couple by $((a'_t), (b'_t))$. By construction, it is again from $\IIC$, and either its size or the maximal element (or both) diminished. If $\tilde{\sigma}((a'_t), (b'_t))$ was chosen to be $w'$, then put
\[w=\tilde{\sigma}((a_t), (b_t)) = \tilde{d}_{b_u-1} \cdots \tilde{d}_{b_k-1} w'.\]
By construction, at each crossing of $w$, a strand starting at a point of the form $(a_t,0)$ goes to the right, and a strand starting at a point not of this form goes to the left. Each strand starting at some $(a_t,0)$ goes to the right at least once, since $a_t < b_t$. Thus $w$ as an $\tilde{n}$-web, its right strands are precisely those starting from the points $(a_t,0)$, and $\tilde{\varrho}(w)=((a_t), (b_t))$. As a consequence, $\tilde{\varrho}\circ\tilde{\sigma} = \Id$.

To show that $\tilde{\sigma}\circ\tilde{\varrho} = \Id$, we will exhibit an isotopy between any $\tilde{n}$-webs $w$ and $v$ sharing the same right sequence couple. This is done by induction on the number of crossings of $v$. Write $v=\tilde{d}_iv'$, where the $\tilde{n}$-web $v'$ has less crossings than $v$. The crossing $\tilde{d}_i$ involves a right strand $r_v$ ending at $(i+1,1)$, and a left strand $l_v$ ending at $(i,1)$. In $w$, there is a right strand $r_w$ with the same endpoints as $r_v$, since $w$ and $v$ share the same right sequence couple. Let its highest (with respect to the $y$-coordinate) crossing $A$ be with a left strand $l_w$. Observe that $l_w$ has no crossings above $A$. Indeed, such crossings are possible only with right strands, the rightmost of which has to end at $(i,1)$, which is then the endpoint of a right strand in $w$ and of the left strand $l_v$ in $v$. But this is impossible for $\tilde{n}$-webs sharing the same right sequence couple. As a result, $l_w$ has to end at $(i,1)$, and the crossing $A$ can be isotoped to the top of $w$, which allows one to write $w=\tilde{d}_iw'$. The induction hypothesis can then be applied to $w'$ and $v'$.
\end{proof}

The adjective ``spiral'' used for the map $\tilde{\sigma}$ comes from the shape of the resulting $\tilde{n}$-webs when the crossing number is big. Following such an $\tilde{n}$-web upwards, one sees the right strands first assemble together in a co-shuffle-like way, then spiral around the core of the cylinder, and then shuffle again with the left strands.

We finish this section with a sketch of a diagrammatic construction of $\tilde{\varrho}^{-1}$, omitting the proofs, which are similar to those given for Propositions \ref{Pr:SeqPerm} and \ref{Pr:SeqPerm_C}. Consider the universal covering $p \colon \RR \times I \to \SSS^1 \times I$ of our cylinder, which identifies points $(x,y)$ and $(x+n,y)$ for all $x \in \RR, y \in I$. Given an $\tilde{n}$-web $w$, take its strand $s$ starting at some $(i,0) \in \SSS^1 \times I$ and consider the unique lift of $s$ to $\RR \times I$ starting at $(i,0) \in \RR \times I$. For a right strand starting at $(a_t,0)$, one gets a strand in $\RR \times I$ ending precisely at $(b_t,1)$. Here $((a_t), (b_t))$ is the right sequence couple of~$w$. A trivial strand lifts to itself. And a left strand starting at $(c_u,0)$ lifts to a strand ending at $(d_u,1)$. Here the sequences $(c_u), (d_u)$ are constructed as follows. The $c_u$ form the ordered sequence of the $x$-coordinates of the lower endpoints of all the left strands of $w$. Next, we put $d_u=s(c_u)-n*g_u$, where $(s(c_u),1)\in \SSS^1 \times I$ is the upper endpoint of the strand starting at $(c_u,0)$, and $g_u$ is the number of times this strand, drawn always going in the negative direction, crosses the segment $\{n+\frac{1}{2}\} \times I$. The sequence couple $((c_u),(d_u))$ is completely determined by $((a_t), (b_t))$. Indeed, define $(c'_u)$ and $(d''_u)$ as the ordered sequences of the elements of $\{1,2,\ldots,n\}$ not belonging to $(a_t)$ and $(b_t \operatorname{mod} n)$ respectively. One gets sequences of size $n-k$. Put $h=\sum_t h_t$, and divide it by $n-k$ to obtain $h=q(n-k)+r$, with $q,r \in \NN\cup\{0\}$ and $r<n-k$. Then put $d'_u= d''_u-n*(q+1)$ for the last $r$ elements of the sequence $(d''_u)$, and $d'_u= d''_u-n*q$ for the remaining ones. Reorder the $d'_u$ to get an increasing sequence. Discard all $c'_u$ and $d'_u$ with $c'_u=d'_u$. The resulting sequences are precisely $((c_u),(d_u))$. They can be thought of as the \emph{left sequence couple} of $w$, and satisfy the conditions defining $\IIC$, with all the inequalities $a_t<b_t$ replaced with $a_t>b_t$. Now, to compute $\tilde{\varrho}^{-1}((a_t), (b_t))$ for any $((a_t), (b_t)) \in \IIC$,
\begin{enumerate}
\item construct the sequences $((c_u),(d_u))$ as described above;
\item draw in $\RR \times I$ straight segments connecting $(a_t,0)$ to $(b_t,1)$ and $(c_u,0)$ to $(d_u,1)$ for all $t$ and $u$;
\item project everything to the cylinder via the map $p$;
\item add vertical segments if necessary (that is, if one gets $c'_u=d'_u$ for some $u$ when constructing $((c_u),(d_u))$).
\end{enumerate}

\begin{exa}\label{EX:4web}
Let us illustrate our algorithm on the case $n=4,\ k=2,\ a_1=2, \ a_2=3, \ b_1=6,\ b_2=9$. We have $6=2+4*1$ and $9=1+4*2$, hence $h_1=1, \ h_2=2, \ h=h_1+h_2=3$. The Euclidean division of $h=3$ by $n-k=2$ yields $3=1*2+1$, hence $q=r=1$. Further, the complements of $\{a_1=2,a_2=3\}$ and $\{b_1 \operatorname{mod} 4 = 2, b_2 \operatorname{mod} 4 = 1\}$ in $\{1,2,3,4\}$ are $\{1,4\}$ and $\{3,4\}$ respectively, so $c'_1=1,\ c'_2=4,\ d''_1=3,\ d''_2=4,\ d'_1=3-4*1=-1,\ d'_2=4-4*2=-4$. The reordering of the $d'_i$ yields $(-4,-1)$. No omissions are necessary since $c'_u>0>d'_u$ for all $u$. Thus the left sequence couple is $((1,4),(-4,-1))$. The graphical part of the algorithm is realised in Fig.~\ref{P:Linearisation_C}. 
\begin{figure}[h]
\centering
\begin{tikzpicture}[x=15,y=15]
\draw [line width=1.2pt,color=qqttzz] (1.,0.)-- (-4.,4.);
\draw [line width=1.2pt,color=wwzzff] (4.,0.)-- (-1.,4.);
\draw [line width=1.2pt,color=ffttcc] (2.,0.)-- (6.,4.);
\draw [line width=1.2pt,color=dcrutc] (3.,0.)-- (9.,4.);
\draw [line width=0.8pt,dash pattern=on 5pt off 5pt] (-7,0.)-- (13,0.);
\draw [line width=0.8pt,dash pattern=on 5pt off 5pt] (-7,4.)-- (13,4.);
\draw [line width=0.8pt,dotted] (5.,0.)-- (5.,4.);
\draw [line width=0.8pt,dotted] (1.,4.)-- (1.,0.);
\draw [line width=0.8pt,dotted] (9.,0.)-- (9.,4.);
\draw [line width=0.8pt,dotted] (-3.,4.)-- (-3.,0.);
\draw[color=ffttcc] (6,4.5) node {$6$};
\draw[color=dcrutc] (9,4.5) node {$9$};
\draw[color=wwzzff] (-1,4.5) node {$-1$};
\draw[color=qqttzz] (-4,4.5) node {$-4$};
\draw[color=qqttzz] (1,-.5) node {$1$};
\draw[color=ffttcc] (2,-.5) node {$2$};
\draw[color=dcrutc] (3,-.5) node {$3$};
\draw[color=wwzzff] (4,-.5) node {$4$};
\draw[color=qqttzz] (5,-.5) node {$5$};
\draw [fill=qqttzz,color=qqttzz] (1.,0.) circle (1.0pt);
\draw [fill=ffttcc,color=ffttcc] (2.,0.) circle (1.0pt);
\draw [fill=dcrutc,color=dcrutc] (3.,0.) circle (1.0pt);
\draw [fill=wwzzff,color=wwzzff] (4.,0.) circle (1.0pt);
\draw [fill=ffttcc,color=ffttcc] (6.,4.) circle (1.0pt);
\draw [fill=dcrutc,color=dcrutc] (9.,4.) circle (1.0pt);
\draw [fill=qqttzz,color=qqttzz] (-4.,4.) circle (1.0pt);
\draw [fill=wwzzff,color=wwzzff] (-1.,4.) circle (1.0pt);
\end{tikzpicture}

\begin{tikzpicture}[x=25,y=25]
\draw (-1.5,2.4) node {$p$};
\draw [line width=1.4pt,->] (-2,2)-- (-1,2);
\draw [line width=0.8pt,dash pattern=on 3pt off 3pt] (1.,0.)-- (5.,0.);
\draw [line width=0.8pt,dotted] (5.,0.)-- (5.,4.);
\draw [line width=0.8pt,dash pattern=on 3pt off 3pt] (5.,4.)-- (1.,4.);
\draw [line width=0.8pt,dotted] (1.,4.)-- (1.,0.);
\draw [line width=1.2pt,color=qqttzz] (4.,4.)-- (5.,3.2);
\draw [line width=1.2pt,color=ffttcc] (5.,3.)-- (2.,0.);
\draw [line width=1.2pt,color=wwzzff] (5.,2.4)-- (3.,4.);
\draw [line width=1.2pt,color=dcrutc] (5.,1.3333333333333333)-- (3.,0.);
\draw [line width=1.2pt,color=qqttzz] (5.,0.)-- (1.,3.2);
\draw [line width=1.2pt,color=ffttcc] (2.,4.)-- (1.,3.);
\draw [line width=1.2pt,color=wwzzff] (1.,2.4)-- (4.,0.);
\draw [line width=1.2pt,color=dcrutc] (5.,4.)-- (1.,1.3333333333333333);
\draw [line width=0.8pt] (4.5,0.)-- (4.5,4.);
\draw (4,4.5) node[anchor=north west] {$x=4.5$};
\draw[color=qqttzz] (1,-.5) node {$1$};
\draw[color=ffttcc] (2,-.5) node {$2$};
\draw[color=dcrutc] (3,-.5) node {$3$};
\draw[color=wwzzff] (4,-.5) node {$4$};
\draw[color=qqttzz] (5,-.5) node {$5$};
\draw [fill=qqttzz,color=qqttzz] (1.,0.) circle (1.0pt);
\draw [fill=ffttcc,color=ffttcc] (2.,0.) circle (1.0pt);
\draw [fill=dcrutc,color=dcrutc] (3.,0.) circle (1.0pt);
\draw [fill=wwzzff,color=wwzzff] (4.,0.) circle (1.0pt);
\draw [fill=ffttcc,color=ffttcc] (2.,4.) circle (1.0pt);
\draw [fill=dcrutc,color=dcrutc] (1.,4.) circle (1.0pt);
\draw [fill=qqttzz,color=qqttzz] (4.,4.) circle (1.0pt);
\draw [fill=wwzzff,color=wwzzff] (3.,4.) circle (1.0pt);
\end{tikzpicture}
\caption{Reconstructing a $\tilde{4}$-web from its right sequence couple $((2,3),(6,9))$}\label{P:Linearisation_C}
\end{figure}
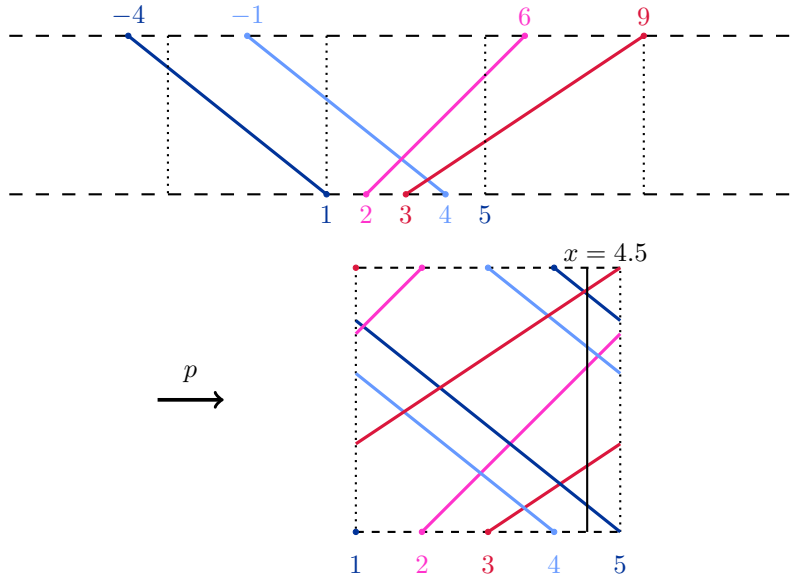
No vertical segments are necessary. Reading this diagram from bottom to top, one sees that the corresponding element of $\HKC$ is 
\[x=x_4x_3x_1x_4x_2x_1x_3x_2x_4x_3.\]
\end{exa}

\newpage 
\section{Main result}

We are now ready to prove all the bijections announced in the Introduction.

\begin{thm}\label{T:Main}
Fix an $n \in \NN$. The maps $\varepsilon$, $\iota$, and $\rho\circ\pi$ described above yield bijections between
\begin{enumerate}[label=(\roman*)]
\item the elements of the Hecke--Kiselman monoid $\HKL$;
\item the set $\WHKL$ of $n$-webs;
\item the set $\II$ of increasing couples of increasing integer sequences bounded by $1$ and $n+1$.
\end{enumerate}
\end{thm}

Most parts of the theorem were treated in Propositions \ref{Pr:HKvsDHK}, \ref{Pr:Disentangle}, and \ref{Pr:SeqPerm}. The bijections and surjections established there are summarised in the following diagram:
\[\xymatrix@!0 @R=1.2cm @C=2.5cm{
\HKL \ar[rr]^{\varepsilon}_{1:1} \ar@{.>}[d]_{\varphi}  &&\DHKL\\
\II &\Perm \ar[l]_{\rho}^{1:1} &\WHKL \ar@{->>}[u]_{\iota} \ar[l]_{\pi}^{1:1}
}\]
The algebraic/diagrammatic part of the story is on the left/right respectively.

We will now describe a map $\varphi \colon \HKL \to \II$ making the above diagram commute, in the sense of $\varphi \circ\varepsilon^{-1} \circ\iota = \rho\circ\pi$. The injectivity, and hence the bijectivity, of $\iota$ follow and complete the proof of the theorem. As a by-product, we obtain an explicit bijection $\varphi$, which can be realised as an elementary algorithm linear in $\max\{n, \operatorname{length}(x)\}$, where $x \in \HKL$ is seen as a word in the generators $x_i$.

The map $\varphi$ is constructed using the $\HKL$-chain $\sigma_i(a,b)=(a,a)$ on the set $\NN$; cf. Example~\ref{EX:NiceAction}. Let us look at how an $x \in \HKL$ acts on $(1,2,\ldots, n+1) \in \NN^{n+1}$. Each generator $x_i$ either has no effect, or propagates some element $a$ to the right, i.e., replaces the right neighbour $b>a$ of some $a$ by $a$. Now, consider only the $a$'s present in the result $y=x \cdot (1,2,\ldots, n+1)$ of this action. Order them to get the sequence $(a_t)$. Denote by $b_t$ the place of the rightmost occurrence of $a_t$ in $y$. Throw away the couples with $b_t=a_t$. Since the propagation happens only to the right, we get a sequence couple $((a_t),(b_t)) \in \II$, declared to be $\varphi(x)$.

\begin{lem}
One has $\varphi\circ \varepsilon^{-1}\circ \iota = \rho\circ\pi$.
\end{lem}

\begin{proof}
One needs to understand $y=x \cdot (1,2,\ldots, n+1)$ for an $x$ corresponding to an $n$-web $w$. Let $((a_t),(b_t))$ be the non-left sequence couple of $w$, that is, we record the endpoints of right and trivial strands. Then $y$ contains only the $a_t$'s, and the element $a_t$ occupies in $y$ the positions $b_{t-1}+1,\ldots, b_t$ (we put $b_0=0$). Throwing  away the couples with $b_t=a_t$, one gets, on the one hand, the right sequence couple of $w$, which is $\rho\circ\pi(w)$; and, on the other hand, $\varphi(y)= \varphi\circ \varepsilon^{-1} \circ\iota(w)$.
\end{proof}

Similarly, in the case $\HKC$ we get

\begin{thm}\label{T:Main_C}
Fix an $n \in \NN$. The maps $\tilde{\varepsilon}$, $\tilde{\iota}$, and $\tilde{\varrho}$ described above yield bijections between
\begin{enumerate}[label=(\roman*)]
\item the elements of the Hecke--Kiselman monoid $\HKC$;
\item the set $\WHKC$ of $\tilde{n}$-webs;
\item the set $\IIC$ of $n$-close increasing couples of increasing integer sequences.
\end{enumerate}
\end{thm}

Indeed, Propositions \ref{Pr:HKvsDHK_C}, \ref{Pr:Disentangle_C}, and \ref{Pr:SeqPerm_C} yield
\[\xymatrix@!0 @R=1.2cm @C=2.5cm{
&\HKC \ar[r]^{\tilde{\varepsilon}}_{1:1} \ar@{.>}[d]_{\tilde{\varphi}}  \ar@{.>}[dl]_{\tilde{\psi}}  &\DHKC\\
\operatorname{SC}& \IIC \ar@{_{(}->}[l]_{\eta} &\WHKC \ar@{->>}[u]_{\tilde{\iota}} \ar[l]_{\tilde{\varrho}}^{1:1}
}\]
We will now describe a map $\tilde{\psi} \colon \HKC \to \operatorname{SC}$, where $\operatorname{SC}$ is the set of all couples of integer sequences. Denoting by $\eta \colon \IIC \hookrightarrow \operatorname{SC}$ the obvious injection, we will then prove that $\tilde{\psi} \circ \tilde{\varepsilon}^{-1} \circ \tilde{\iota} = \eta \circ \tilde{\varrho}$. The injectivity, and hence the bijectivity, of $\tilde{\iota}$ follow and complete the proof of the theorem. As a by-product, we get $\operatorname{Im}(\tilde{\psi}) \subseteq \IIC$, so $\tilde{\psi}$ induces a map  $\tilde{\varphi} \colon \HKC \to \IIC$ satisfying $\tilde{\varphi} \circ \tilde{\varepsilon}^{-1} \circ \tilde{\iota} = \tilde{\varrho}$. The bijectivity of $\tilde{\iota}$ implies that we obtain an explicit bijection $\tilde{\varphi}$, which can be realised as an elementary algorithm linear in $\max\{n, \operatorname{length}(x)\}$, where $x \in \HKC$ is seen as a word in the generators $x_i$.

The map $\tilde{\psi}$ is constructed using the $\HKC$-chain $\sigma_i(a,b)=(a,a)$ for $i<n$, and $\sigma_n(a,b)=(a,a+n)$, on the set $\NN$; cf. Example~\ref{EX:NiceAction}. Let us look at how an $x \in \HKC$ acts on $(1,2,\ldots, n) \in \NN^{n}$. Each generator $x_i$ either has no effect, or propagates some element $a$ to the right, or replaces the first element with the last element $+n$. Now, consider only those $a \in \{1,2,\ldots,n\}$ which coincide $\operatorname{mod} n$ with at least one entry of $y=x \cdot (1,2,\ldots, n)$. Order them to get the sequence $(a_t)$. For each $a_t$, consider the maximal number $m_t$ coinciding with $a_t$ $\operatorname{mod} n$ and occurring in $y$, and denote by $s(a_t)$ the place of the rightmost occurrence of $m_t$ in $y$. Decompose $m_t$ as $a_t+n*h_t$, and put $b_t=s(a_t)+n*h_t=s(a_t)+m_t-a_t$. Throw away the couples with $b_t=a_t$. Declare the resulting sequence couple $((a_t),(b_t))$ to be $\tilde{\psi}(x)$.

\begin{exa}
Let us evaluate $\tilde{\psi}$ on the element $x= x_4x_3x_1x_4x_2x_1x_3x_2x_4x_3 \in \mathsf{C}_4$ from Example~\ref{EX:4web}. We need to compute $y=x \cdot (1,2,3,4)$:
\begin{align*}
(1,2,3,4) &\overset{x_3}{\mapsto}(1,2,3,3)\overset{x_4}{\mapsto}(7,2,3,3)\overset{x_2}{\mapsto}(7,2,2,3)\overset{x_3}{\mapsto}(7,2,2,2)\\
&\overset{x_1}{\mapsto}(7,7,2,2)\overset{x_2}{\mapsto}(7,7,7,2)\overset{x_4}{\mapsto}(6,7,7,2)\overset{x_1}{\mapsto}(6,6,7,2)\overset{x_3}{\mapsto}(6,6,7,7)\\
&\overset{x_4}{\mapsto}(11,6,7,7)=y.
\end{align*}
Modulo $4$, this yields $(3,2,3,3)$, which contains only $2$ and $3$. Thus $a_1=2, \ a_2=3, \ m_1=6=2+4*1, \ m_2=11=3+4*2,\ h_1=1,\ h_2=2$. Further, $s(a_1)=s(2)$ is the rightmost occurrence of $m_1=6$ in $y$, which is $2$, and $s(a_2)=s(3)$ is the rightmost occurrence of $m_2=11$ in $y$, which is $1$. Finally, $b_1=s(a_1)+m_1-a_1=2+6-2=6, \ b_2=s(a_2)+m_2-a_2=1+11-3=9$. This is the expected result, since in Example~\ref{EX:4web} $x$ was constructed out of the $\tilde{4}$-web $\tilde{\varrho}^{-1}((2,3),(6,9))$.
\end{exa}

\begin{lem}
One has $\tilde{\psi} \circ \tilde{\varepsilon}^{-1} \circ \tilde{\iota} = \eta \circ \tilde{\varrho}$.
\end{lem}

\begin{proof}
One needs to understand $y=x \cdot (1,2,\ldots, n)$ for an $x$ corresponding to an $\tilde{n}$-web $w$. Let $((a_t),(b_t))$ be the non-left sequence couple of $w$, that is, we record the endpoints of right and trivial strands. Modulo $n$, $y$ contains only the $a_t$'s, and the element $a_t$ occupies in $y$ the positions $b_{t-1}+1,\ldots, b_t$ if $t>1$, and $b_{k}+1,\ldots, b_1+n$ if $t=1$. Here $k$ is the size of $(a_t)$. To get from integers modulo $n$ to integers, observe that in the propagation story the $+n$ phenomenon happens only when an element propagates from the last position to the first one, that is, when the generator $x_n$ is applied, which in an $\tilde{n}$-web corresponds to the elementary diagram $\tilde{d}_n$, which is the only case when a right strand crosses the $x=n+\frac{1}{2}$ line, which in its turn is the only situation when the corresponding $b_t$ is augmented by $n$. Thus, throwing  away the couples with $b_t=a_t$ as usual, one gets, on the one hand, the right sequence couple of $w$, which is $\tilde{\varrho}(w)$; and, on the other hand, $\tilde{\psi}(y)= \tilde{\psi}\circ \tilde{\varepsilon}^{-1} \circ \tilde{\iota}(w)$.
\end{proof}

\bigskip
\subsection*{Acknowledgments}
This paper originated from discussions with Magdalena Wiertel during her stay in Caen. The author is grateful to her for a detailed introduction into the subject.

\bibliographystyle{alpha}
\bibliography{refs}
\end{document}